\documentclass[11pt]{amsart}
\usepackage{amsmath,amsthm,amssymb,amscd,cmbright}
\usepackage{amsrefs}
\usepackage[T1]{fontenc}
\usepackage[utf8]{inputenc}
\usepackage{lmodern}
\usepackage[hypertex]{hyperref}
\usepackage{amsfonts}
\usepackage{graphicx}
\usepackage{graphics}
\usepackage[english]{babel}
\usepackage{hhline}
\usepackage[dvips]{color}
\usepackage{pb-diagram}
\DeclareGraphicsExtensions{.png,.pdf,.jpg,.gif}
\usepackage{youngtab}
\usepackage{tikz}
\usetikzlibrary{chains}

\newtheorem{theorem}{Theorem}[section]
\newtheorem{teor}{Theorem}[section]

\newtheorem{corol}[theorem]{Corollary}
\newtheorem{lemma}[teor]{Lemma}

\theoremstyle{definition}
\newtheorem{rk}[teor]{Remark}

\newcommand{\R}{\mbox{${\mathbb R}$}}
\newcommand{\C}{\mbox{${\mathbb C}$}}

\newcommand{\Z}{\mbox{${\mathbb Z}$}}

\newcommand{\w}{\mbox{${\omega}$}}

\newcommand{\GW}{\mbox{$\operatorname{GW}$}}

\DeclareFontFamily{U}{mathx}{\hyphenchar\font45}
\DeclareFontShape{U}{mathx}{m}{n}{
      <5> <6> <7> <8> <9> <10>
      <10.95> <12> <14.4> <17.28> <20.74> <24.88>
      mathx10
      }{}
\DeclareSymbolFont{mathx}{U}{mathx}{m}{n}
\DeclareFontSubstitution{U}{mathx}{m}{n}
\DeclareMathAccent{\widecheck}{0}{mathx}{"71}
\DeclareMathAccent{\wideparen}{0}{mathx}{"75}

\begin{document}

\title{\bf Hofer-Zehnder capacity and Bruhat graph}
\author{Alexander Caviedes Castro}

\email{alexanderc1@post.tau.ac.il}

\date{\today}
\maketitle

\begin{abstract}
We find bounds for the Hofer-Zehnder capacity of coadjoint orbits of
compact Lie groups with respect to the Kostant--Kirillov--Souriau
symplectic form  in terms of the combinatorics of their Bruhat
graph. We show that our bounds are sharp for coadjoint orbits of the
unitary group and equal to the diameter of a weighted Cayley graph.
\end{abstract}

\section{Introduction}

The Gromov non-squeezing theorem in symplectic geometry states that
is not possible to embed symplectically a ball into a cylinder of
smaller radius, although this can be done with volume preserving
embeddings \cite{gromov}. Hence, the biggest radius of a ball that
can be symplectically embedded into a symplectic manifold can be
used as a way to measure the ``symplectic size'' of the manifold. We
call this radius \textit{Gromov's width}.

The Gromov width as a symplectic invariant is extended through the
notion of \textit{symplectic capacity} whose axiomatic formulation
is due to Ekeland and Hofer \cite{EkelandHoferI},
\cite{EkelandHoferII}. An important example of capacity is the
\textit{Hofer-Zehnder capacity} \cite{HoferZehnder}. The
Hofer-Zehnder capacity of a closed symplectic manifold $(M, \w)$ is
defined as
$$
\operatorname{c_{HZ}}(M, \w):=\sup{\Bigl
\{\operatorname{max}{H}-\operatorname{min}{H}\,\,|\,\, H:M\to\R
\text{ slow} \Bigr\}},
$$
where a Hamiltonian $H:M\to \R$ is slow if the periodic trajectories
of its Hamiltonian flow are either constant or have period greater
or equal to one. In comparison with the Gromov width, the
Hofer-Zehnder capacity measures the size of a symplectic manifold in
a Hamiltonian dynamic way

In this paper, we are interested in computing bounds for the
Hofer-Zehnder capacity of coadjoint orbits of compact Lie groups
with respect to their Kostant-Kirillov-Souriau form. We summarize
the main results in this paper in the following theorem.

\begin{teor}

Let $G$ be a compact connected simple Lie group with Lie algebra
$\mathfrak{g}.$ We identify the Lie algebra $\mathfrak{g}$ with its
dual $\mathfrak{g}^*$ via an adjoint invariant inner product. Let
$T\subset G$ be a maximal torus. For $\lambda \in \mathfrak{t}
\subset \mathfrak{g},$ let $\mathcal{O}_\lambda$ be the coadjoint
orbit passing through $\lambda$ and $\w_\lambda$ be the
Kostant--Kirillov--Souriau form defined on $\mathcal{O}_\lambda.$
Let $R$ be the corresponding system of roots and $S$ be a choice of
simple roots.

For a positive root $\beta$ we write
$$
\beta=\sum_{\alpha\in S}n_{\beta\alpha}\alpha
$$
for some nonnegative integer $n_{\beta\alpha}.$ We denote by
$\check{\beta}$ the coroot associated with a root $\beta.$

Let $W=N_G(T)/T$ be the Weyl group relative to $T$ and $w_0$ be the
longest element in $W$ relative to the set of simple roots $S.$ If
there exist positive roots $\alpha_1, \cdots, \alpha_r$ such that
$$
w_0=s_{\alpha_1}\cdot \ldots \cdot s_{\alpha_r},
$$
then we obtain the following bounds for the Hofer-Zehnder capacity
of $(\mathcal{O}_\lambda, \w_\lambda)$
\begin{equation}\label{inequality}
\max_{\alpha \in S}\Bigl\{\sum_{k=1}^r
\dfrac{n_{\alpha_k\,\alpha}}{n_{\rho\,\alpha}}|\langle \lambda,
\check{\alpha}_k \rangle|  \Bigr\} \leq
c_{\operatorname{HZ}}(\mathcal{O}_\lambda, \w_\lambda) \leq
\sum_{k=1}^r |\langle \lambda, \check{\alpha}_k \rangle|,
\end{equation}
here $\rho$ denotes the highest positive root.
\end{teor}

In the proof of the nonsqueezing theorem, Gromov noted that the
Gromov width of a symplectic manifold is constrained by the
existence of pseudoholomorphic curves \cite{gromov}. The relation
between pseudoholomorphic curves and the Hofer-Zehnder capacity was
also observed by several authors in the context of the Weinstein
conjecture (see e.g  Floer, Hofer and Viterbo
\cite{FloerHoferViterbo}, Hofer and Viterbo \cite{HoferViterbo}, Liu
and Tian \cite{LiuTian}). This relation appears more explicit in a
result of G. Lu that bounds the Hofer-Zenhder capacity of a
symplectic manifold when it has a nonzero Gromov-Witten invariant
with two point constrains \cite{Lu}, \cite{glu2}. In this paper, we
use G. Lu's result to bound from above the Hofer-Zehnder capacity of
coadjoint orbits of compact Lie groups.

The \textit{Bruhat graph} (also known as \textit{moment graph} or
\textit{GKM graph}) of a coadjoint orbit is the graph whose vertices
and edges are in one to one correspondence with the points and
irreducible invariant curves that are invariant with respect to the
action of a maximal torus on the coadjoint orbit. A result of Fulton
and Woodward states that the minimal degrees appearing in the
nonvanishing Gromov-Witten invariants of a coadjoint orbit can be
interpreted in terms of paths of its Bruhat graph \cite{fultonw}.
The main goal of the present paper is to point out the relation
between the Bruhat graph and the Hofer-Zehnder capacity of coadjoint
orbits of compact Lie groups.

If we weight the edges of the Bruhat graph of the coadjoint orbit
with the symplectic area of the curves that they represent, then the
right hand side of the inequality \eqref{inequality} appearing in
the Main theorem can be reinterpreted as the following inequality
\[
c_{\operatorname{HZ}}(\mathcal{O}_\lambda, \w_\lambda)\leq
\operatorname{diameter\,\, weighted \,\,Bruhat\,\, graph\,\, of \,\,
}(\mathcal{O}_\lambda, \w_\lambda)
\]
In this paper we show that the previous inequality is sharp for
coadjoint orbits of the unitary group.

\begin{teor}
Let $\lambda=(\lambda_1, \cdots, \lambda_n)\in \R^n$ and assume that
$\lambda_1\geq\cdots \geq \lambda_n.$ Let
$$
\mathcal{H}_\lambda:=\{A\in M_n(\C): A^*=-A, \operatorname{spectrum
}{A}=i\lambda\}.
$$
We identify $\mathcal{H}_\lambda$ with a coadjoint orbit of $U(n)$
and endow it with a symplectic form $\w_\lambda$ coming from the
Kostant-Kirillov-Souriau form.

Let us consider the weighted Cayley graph of the symmetric group
$S_n$ where two permutations are joined by an edge of weight
$|\lambda_i-\lambda_j|$ if they differ by a trasposition $(i, j).$
Then
\begin{align*}
c_{\operatorname{HZ}}(\mathcal{H}_\lambda, \w_\lambda)&=
\dfrac{1}{2}\sum_{k=1}^n|\lambda_k-\lambda_{n-k+1}|\\&=\operatorname{diameter
\,\,of\,\, the\,\, weighted\,\, Cayley\,\, graph\,\, of \,\,} S_n.
\end{align*}
\end{teor}
The Hofer-Zehnder capacity of a coadjoint orbit of the unitary group
is in contrast with its Gromov width that is equal to the smallest
weight of the weighted Cayley graph of $S_n$ defined in the previous
theorem (see e.g. Caviedes \cite{caviedes}, Pabiniak
\cite{pabiniak}). In particular, the Hofer-Zehnder capacity of a
coadjoint orbit isomorphic with a projective space coincides with
its Gromov width and the Hofer-Zehnder capacity of a coadjoint orbit
isomorphic with a Grassmannian manifold is equal to an integer
multiple of its Gromov's width, and we recover results of Hofer and
Viterbo for the projective space \cite{HoferViterbo} and G. Lu for
the Grassmannian manifold \cite{Lu}.

We suggest that the reader compares our results with the ones of
Loi, Mossa and Zuddas \cite{LoiMossaZuddas} where they estimate the
Hofer-Zenhder capacity of Hermitian symmetric spaces and with the
ones of Hwang and Suh \cite{HwangSuh} where they compute the
Hofer-Zehnder capacity of symplectic manifolds with Hamiltonian
semifree circle actions in terms of their moment map.

This paper is organized as follows: in the second section, we review
the definition of Hofer-Zehnder capacity of a symplectic manifold
and state G.Lu's theorem that bounds the Hofer-Zehnder capacity of a
symplectic manifold in terms of its Gromov-Witten invariants. In the
third section, we recall background on the geometry of coadjoint
orbits of compact Lie groups.  In the fourth section, we define the
Bruhat graph and indicate its relation with the Hofer-Zehnder
capacity of coadjoint orbits. In the fifth section, we compute the
Hofer-Zehnder capacity of coadjoint orbits of the unitary group. In
the sixth section, we recall results of Postnikov concerning the
minimal degrees of paths in the Bruhat graph, and explain how they
can be used to find more optimal upper bounds for the Hofer-Zehnder
capacity of regular coadjoint orbits. In the seventh section, we
explain how to bound from below the Hofer-Zehnder capacity of a
coadjoint orbit using the moment map of the Hamiltonian group action
of a maximal torus. In the eight section we write explicitly our
bounds for every simple compact Lie group according to the type.

\section{Hofer-Zehnder capacity and Gromov-Witten invariants}

Let $(M, \w)$ be a closed symplectic manifold. A
\textit{Hamiltonian} is a smooth function $H:(\R/\Z)\times M \to
\R.$ A Hamiltonian is \textit{autonomous} if it is time-independent.
The Hamiltonian vector field of $H$ is the time-dependent vector
field $X_H$ defined by
$$
d(H(t, \cdot))=\iota_{X_{H(t, \cdot)}}\w.
$$
The \textit{oscillation} of an autonomous Hamiltonian $H:M\to \R$ is
$$
\operatorname{osc}{H}:=\max{H}-\min{H}
$$
A Hamiltonian function $H:M\to \R $ is \textit{slow} if all periodic
orbits of the Hamiltonian vector field $X_H$ of period less than one
are constant.

The \textit{Hofer-Zehnder capacity} of $(M, \w)$ is defined as
$$
\operatorname{c_{HZ}}(M, \w):=\sup{\bigl
\{\operatorname{osc}{H}\,\,|\,\, H:M\to\R \text{ slow} \bigr\}}
$$
Let $J$ be a Fredholm regular almost complex structure compatible
with $\w,$ (see the definition for instance in McDuff and Salamon
\cite{mcduff}). Let $d$ be a class in $H_2(M;\Z).$ We denote by
$\operatorname{GW}_{d, k}(a_1, a_2, \cdots, a_k)$ the
\textit{Gromov-Witten invariant} that roughly speaking counts the
number of $J$-holomorphic spheres in $M$ in the class $d \in H_2(M;
\Z)$ that meet cycles representing the homology classes $a_1, a_2,
\cdots, a_k \in H_*(M, \Z).$ The following theorem due to G. Lu
bounds from above the Hofer-Zehnder capacity of a closed symplectic
manifold in terms of its Gromov-Witten invariants \cite{Lu}.

\begin{teor}[G. Lu \cite{Lu}]\label{Lu}
Let $(M, \w)$ be a closed symplectic manifold. Suppose that $(M,
\w)$ admits a nonzero Gromov-Witten invariant of the form
$$
\operatorname{GW}_{d,k}([\operatorname{pt}], [\operatorname{pt}],
a_2, \cdots, a_k)
$$
for some $k\in \Z_{\geq 1}, d\in H_2(M; \Z)$ and $a_2, \cdots,
a_k\in H_*(M; \Z).$ Then
$$
\operatorname{c_{HZ}}(M, \w)\leq \w(d\,)
$$
\end{teor}

\section{Geometry of Coadjoint orbits}\label{coadjointorbits}

In this section we establish the Lie theoretical convention that is
used through the rest of the paper. Most of the material can be
found in the classical literature that is concerned about the
geometry and topology of coadjoint orbits such as Bernstein, Gelfand
and Gelfand \cite{GelfandBernstein} and Kirillov \cite{Orbit}.

Let $G$ be a compact Lie group, $\mathfrak{g}$ be its Lie algebra
and $\mathfrak{g}^*$ be the dual of $\mathfrak{g}$. Let $(\cdot\, ,
\cdot)$ denote an adjoin invariant inner product defined on
$\mathfrak{g}.$ We identify the Lie algebra $\mathfrak{g}$ and its
dual $\mathfrak{g}^*$ via this inner product. Let $\lambda\in
\mathfrak{g}^*$ and $\mathcal{O}_\lambda\subset \mathfrak{g}^*$ be
the coadjoint orbit passing through $\lambda.$ Let $\w_\lambda$ be
the \textit{Kostant-Kirillov-Souriau form} defined on
$\mathcal{O}_\lambda$ by
$$
\w_{\lambda}(\hat{X}, \hat{Y})=\langle\lambda, [X, Y] \rangle \ \ \
X, Y \in \mathfrak{g},
$$
where $\hat{X}, \hat{Y}$ are the vector fields on $\mathfrak{g}^*$
generated by the coadjoint action of $G.$ The form $\w_\lambda$ is
closed and non-degenerate thus defining a symplectic structure on
$\mathcal{O}_\lambda.$

We denote by $G_{\mathbb{C}}$ the complexification of the Lie group
$G.$ Let $P\subset G_{\mathbb{C}}$ be a parabolic subgroup of
$G_{\mathbb{C}}$ such that $\mathcal{O}_\lambda \cong
G_{\mathbb{C}}/P.$ The quotient of complex Lie groups
$G_{\mathbb{C}}/P$ allows us to endow $\mathcal{O}_\lambda$ with a
complex structure $J$ compatible with $\w_\lambda$ so the triple
$(\mathcal{O}_\lambda, \w_\lambda, J)$ is a K\"ahler manifold. The
almost complex structure $J$ is Fredholm regular (see e.g. McDuff
and Salamon \cite{mcduff}*{Proposition 7.4.3}).

Let $T\subset G$ be a maximal torus and $\mathfrak{t}$ denote its
Lie algebra. Let $R \subset \mathfrak{t}^*$ be the root system of
$T$ in $G$ so
$$
\mathfrak{g}_{\mathbb{C}}=\mathfrak{t}_{\mathbb{C}} \oplus
\bigoplus_{\alpha \in R}\mathfrak{g}_\alpha,
$$
where
$$
\mathfrak{g}_\alpha:=\{x\in \mathfrak{g}_{\mathbb{C}}:[\,h\,,
x\,]=\alpha(h)\,x \, \text{ for all }h\in
\mathfrak{t}_{\mathbb{C}}\}$$ is the root space associated with the
root $\alpha \in R.$ Let $R^{+}\subset R$ be a choice of positive
roots with simple roots $S \subset R^+.$ Let $B\subset
G_{\mathbb{C}}$ be the Borel subgroup with Lie algebra
$$
\mathfrak{b}=\mathfrak{t}_{\mathbb{C}}\oplus \bigoplus_{\alpha \in
R^+}\mathfrak{g}_{\alpha}
$$

Each root $\alpha \in R$ has a coroot $\check{\alpha}\in
\mathfrak{t}.$ The coroot $\check{\alpha}$ is identified with
$\frac{2\alpha}{(\alpha, \alpha)}\in \mathfrak{t}$ via the invariant
inner product $(\cdot\,,\cdot).$ The system of coroots is the set
$\check{R}=\{\check{\alpha}:\alpha \in R\}$ and the simple coroots
is the set $\check{S}=\{\check{\alpha}:\alpha \in S\}.$

Every root $\alpha \in R$ defines a reflection $s_\alpha$ on
$\mathfrak{t}^*$ given by
\begin{align*}
s_\alpha:\mathfrak{t}^* &\to \mathfrak{t}^*\\
t  &\mapsto t-\langle t, \check{\alpha}\rangle\, \alpha,
\end{align*}
where $\langle\cdot\, ,\cdot\, \rangle$ denotes the standard pairing
$\langle\cdot\, , \cdot\,\rangle:\mathfrak{t}^*\otimes \mathfrak{t}
\to \R.$ The group $W$ generated by the set of reflections
$\{s_\alpha\}_{\alpha\in R}$ is the \textit{Weyl group} of $G$
relative to $T.$ It is known that the Weyl group $W$ can be
identified with $N_G(T)/T.$

We can associate to the parabolic subgroup $P\subset G_{\mathbb{C}}$
a subset of simple roots defined by
$$
S_P:=\{\alpha\in S: \langle \lambda, \check{\alpha}\rangle\ne 0\}.
$$
We denote by $W_P$ the Weyl group of $P$ generated by the set of
simple roots $S_P.$ The Weyl group $W_P$ is identified with
$N_P(T)/T.$ Also, set $R_P:=R \cap \Z S_P$ and  $R_P^+:=R^+ \cap \Z
S_P,$ where $\Z S_P=\operatorname{span}_{\mathbb{Z}}(S_P).$

The \textit{Weyl chamber} relative to the set of simple roots $S$ is
the convex polyhedron
$$
\mathfrak{t}_+^*:=\{\gamma\in \mathfrak{t}^*: \langle \gamma,
\check{\alpha}\rangle \geq 0 \text{ for all }\alpha\in S\}
$$
The vector space $\mathfrak{t}^*$ can be decomposed as the union of
convex polyhedrons
$$
\mathfrak{t}^*=\bigcup_{w\in W}w\bigl( \mathfrak{t}_+^*\bigr)
$$
whose interiors are disjoint. For the coadjoint orbit
$\mathcal{O}_\lambda,$ there exists $\lambda'\in \mathfrak{t}_+^*$
such that $\mathcal{O}_\lambda\cap
\mathfrak{t}^*=\{w(\lambda')\}_{w\in W}.$ Indeed, the group $T$ acts
hamiltonially on $\mathcal{O}_\lambda $ with moment map
$\phi:\mathcal{O}_\lambda \hookrightarrow \mathfrak{t}^*$ equals to
the composition of the projection map $\mathfrak{g}^*\to
\mathfrak{t}^*$ with the inclusion map $\mathcal{O}_\lambda
\hookrightarrow \mathfrak{g}^*.$ The image of $\phi$ is the convex
hull of $\{w(\lambda)\}_{w\in W}.$ We always assume in what follows
that $\lambda\in \mathfrak{t}_+^*.$

For $w\in W,$ the \textit{length} $l(w)$ of $w$ is defined as the
minimum number of simple reflections $s_\alpha\in W, \alpha\in S,$
whose product is $w.$ The Weyl group $W$ has a unique longest
element that we denote by $w_0.$

Let $B^{\operatorname{op}}:=w_0Bw_0\subset G_{\mathbb{C}}$ be the
\textit{Borel subgroup opposite} to $B.$ For $w\in W/W_P,$ let
$X(w):=\overline{BwP/P} \subset G_{\mathbb{C}}/P$ and
$Y(w):=\overline{B^{\operatorname{op}}wP/P} \subset
G_{\mathbb{C}}/P$ be the \textit{Schubert variety} and the
\textit{opposite Schubert variety} associated with $w,$
respectively. We denote by $\sigma_w$ and $\check{\sigma}_{w}$ the
fundamental classes in the homology group $H_*(G_{\mathbb{C}}/P;
\Z)$ of $Y(w)$ and $X(w),$ respectively. Note that
$\check{\sigma}_{w}=\sigma_{w_ow}=\sigma_{\check{w}},$ where
$\check{w}:=w_0w.$ The set of Schubert classes $\{\sigma_w\}_{w\in
W/W_P}$ forms a free $\Z$-basis of $H_*(G_{\mathbb{C}}/P; \Z),$ and
the set of Schubert classes $\{\check{\sigma}_{w}\}_{w\in W/W_P}$ is
the dual basis of $\{\sigma_w\}_{w\in W/W_P}$ with respect to the
Poincar\'e intersection pairing.

The \textit{Bruhat order} $\prec$ on $W/W_P$ is defined by $u \prec
v$ if $X(u)\subset X(v).$

\section{Bruhat Graph}

In this section we define the Bruhat graph and indicate its relation
with the Hofer-Zehnder capacity of a coadjoint orbit of compact Lie
group.

We keep the convention of the last section. Let $G$ be a compact Lie
group. Let $T\subset G$ be a maximal torus, $B\subset
G_{\mathbb{C}}$ be a Borel subgroup and $P\subset G_{\mathbb{C}}$ be
a parabolic subgroup such that $T\subset B \subset P.$ Let $R$ and
$S$ be the system of roots and simple roots determined by $T$ and
$B,$ respectively.


The \textit{Bruhat graph} is the graph on $W/W_P$ where two elements
are joined by an edge if they differ by one reflection. More
precisely, there is an edge joining $u$ with  $v$ if and only if
there exists a positive root $\alpha \in R^+-R^+_P$ such that
$$
v=u\cdot s_\alpha \mod{W_P}.
$$
In Figure \ref{bruhats3}, we show the Bruhat graph of the Weyl group
of $U(3).$ The standard set of simple roots of $U(3)$ is
$\{\alpha_1=e_1-e_2, \alpha_2=e_2-e_3\}\subset \R^3,$ where $\{e_1,
e_2, e_3\}$ denotes the standard basis of $\R^3.$ The Weyl group of
$U(3)$ is the symmetric group $S_3$ generated by the simple
reflections $s_1:=s_{e_1-e_2}=(1\,2), s_2:=s_{e_2-e_3}=(2\, 3).$

\begin{figure}[h!]
\centering
\includegraphics[scale=0.5]{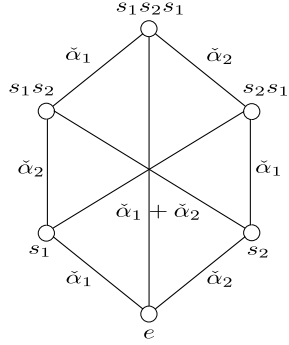}
\caption{Bruhat graph of $S_3$}\label{bruhats3}
\end{figure}

The vertices and edges of the Bruhat graph are in one-to-one
correspondence with the $T$-fixed points and irreducible
$T$-invariant curves of $G_{\mathbb{C}}/P,$ respectively. The
collection of cosets $\{wP\}_{w\in W/W_P}$ is the set of all
$T$-fixed points of $G_{\mathbb{C}}/P.$ For each positive root
$\alpha \in R^+-R^+_P$ there is a unique irreducible $T$-invariant
curve $C_\alpha$ that contains $1\cdot P$ and $s_\alpha\cdot P.$
Indeed, $C_\alpha:=\operatorname{Sl}(2, \C)_\alpha\cdot P/P$ where
$\operatorname{Sl}(2, \C)_\alpha\subset G_{\mathbb{C}}$ is the
subgroup of $G_{\mathbb{C}}$ with Lie algebra
$$
\mathfrak{g}_\alpha\oplus \mathfrak{g}_{-\alpha}\oplus
[\mathfrak{g}_\alpha,\mathfrak{g}_{-\alpha}]\subset
\mathfrak{t}^*_{\mathbb{C}}.
$$
We identify the homology group $H_2(G_{\mathbb{C}}/P; \Z)$ with
$\Z\check{S}/\Z\check{S}_P$ via the transformation
$$
[C_\alpha]\mapsto \check{\alpha}+\Z S_P
$$
(see e.g. Fulton and Woodward \cite{fultonw}). We weight the edges
of the Bruhat graph with elements in $H_2(G_{\mathbb{C}}/P; \Z)\cong
\Z \check{S}/\Z \check{S}_P.$ If $u$ and $v$ differ by the
reflection $s_\alpha,$ then the weight of the corresponding edge is
$\check{\alpha}+\Z \check{S}_P.$ We define an ordering on the set of
degrees $H_2(G_{\mathbb{C}}/P; \Z)$ as follows: for $c, d\in
H_2(G_{\mathbb{C}}/P; \Z),$ we say that $c\leq d$ if there exists
$n_\alpha \in \Z_{\geq 0}$ such that
$$
d-c =\sum_{\alpha\in S-S_P}n_\alpha\check{\alpha}\mod \Z S_P
$$
A \textit{chain} from $u$ to $v$ in $W/W_P$ is a sequence $u_0, u_1,
\cdots, u_r \in W/W_P$ such that $u_{i}$ and $u_{i-1}$ are adjacent
for $1\leq i \leq r,$ $u \prec u_0$ and $u_r \prec \check{v}=w_0v.$
A chain from $u$ to $v$ corresponds to a sequence of $T$-invariant
curves $C_1, C_2, \cdots, C_r$ with $C_1$ meeting $Y(u)$ and $C_r$
meeting $X(\check{v}).$ The \textit{degree} of the chain is the sum
of the classes $[C_i]$ of the curves. A \textit{path} from $u$ to
$v$ in $W/W_P$ is a sequence $u_0, u_1, \cdots, u_r \in W/W_P$ such
that $u_{i}$ and $u_{i-1}$ are adjacent for $1\leq i \leq r,$
$u=u_0$ and $u_r=v.$ A path in the Bruhat graph coincides with the
standard notion of path in graph theory. The degree of a path is
defined in the same way as the degree of a chain.

The following result due to Fulton and Woodward establishes the
relation between chains in the Bruhat graph of $W/W_P$ and the
Gromov-Witten invariants of $G_{\mathbb{C}}/P.$

\begin{teor}[Fulton-Woodward \cite{fultonw}]\label{upperbound}
Let $u, v\in W/W_P$ and $d\in H_2(G_{\mathbb{C}}/P; \Z).$ The
following are equivalent:

\begin{enumerate}

\item There is a chain of degree $c\leq d$ between $u$ and $v$ in the
Bruhat graph of $W/W_P.$

\item There exists a sequence $C_0, C_1, \cdots, C_r$ of $T$-invariant
curves with $C_0$ meeting $Y(u)$ and $C_r$ meeting $X(\check{v}),$
with $C_{i-1}$ meeting $C_i$ for $1\leq i \leq r,$ and with
$\sum_{i=0}^r [C_i]\leq d.$

\item There exist a degree $c\leq d$ and $w$ in
$W/W_P$ such that
$$
\GW_{c,3}(\sigma_u, \sigma_v, \sigma_w)\ne 0
$$
\end{enumerate}
\end{teor}

Now we state the relation between chains in the Bruhat graph and the
Hofer-Zehnder capacity of coadjoint orbits.

\begin{teor}\label{Ca}
Assume that $\lambda$ lies in the Weyl chamber $\mathfrak{t}^*_+$
relative to the set of simple roots $S$ and the coadjoint orbit
$\mathcal{O}_\lambda$ passing through $\lambda$ is isomorphic with
$G_{\mathbb{C}}/P.$ We endow the coadjoint orbit
$\mathcal{O}_\lambda$ with its Kostant-Kirillov-Souriau symplectic
form $\w_\lambda$. Then
$$
\operatorname{c_{HZ}}(\mathcal{O}_\lambda, \w_\lambda)\leq \min_d\,
\w_\lambda(d),
$$
where the minimum is taken over all the degrees $d\in
H_2(\mathcal{O}_\lambda; \Z)\cong \Z\check{S}/\Z\check{S}_P$ of
\textit{paths} joining $[e\,]$ with $[w_0]$ in the Bruhat graph of
$W/W_P.$
\end{teor}
\begin{proof}
Let $d$ be minimal among the set of all degrees of paths joining
$P/P$ with $w_0P/P$ in the Bruhat graph. 
According to the Theorem \ref{upperbound} of Fulton and Woodward,
there exists $u\in W/W_P$ such that
$$
\operatorname{GW}_{d, 3}([\operatorname{pt}], [\operatorname{pt}],
\sigma_{u})\ne 0
$$
By Theorem \ref{Lu} of G.Lu,
$$
\operatorname{c_{HZ}}(\mathcal{O}_\lambda, \w_\lambda) \leq
\w_\lambda(d),
$$
and we are done.
\end{proof}

\begin{rk}
A path in the Bruhat graph joining $P/P$ with $w_0P/P$ is the same
as an ordered sequence of positive roots $\alpha_1, \cdots, \alpha_r
\in R-R_P$ such that
$$
w_0= s_{\alpha_1}\cdot \ldots \cdot s_{\alpha_r} \mod W_P
$$
In this case,  the path in the Bruhat graph is given by the sequence
\begin{align*}
e &\xrightarrow{\alpha_1} s_{\alpha_1} \xrightarrow{\alpha_2}
s_{\alpha_1}s_{\alpha_2} \xrightarrow{\alpha_3}\ldots
\xrightarrow{\alpha_{r-1}} s_{\alpha_1}\cdot \ldots \cdot
s_{\alpha_{r-1}}\xrightarrow{\alpha_{r}}  w_0
\end{align*}
The degree of the corresponding sequence of $T$-invariant curves is
equal to
$$
\sum_{i=1}^r \check{\alpha_i}
$$
The symplectic area of the curve $C_\alpha$ with respect to the
Kostant-Kirillov-Souriau form is equal to $\langle \lambda,
\check{\alpha} \rangle$ (see e.g. McDuff and Tolman
\cite{McDuffTolman}[Lemma 3.9]), hence the symplectic area of the
above sequence of $T$-invariant curves is equal to
$$
\sum_{i=1}^r \langle \lambda, \check{\alpha_i} \rangle.
$$
\end{rk}

\section{Hofer-Zehnder capacity of coadjoint orbits of $U(n)$}

In this section we compute the Hofer-Zehnder capacity of coadjoint
orbits of the unitary group.

We denote by $U(n)$ the set of $n\times n$ unitary matrices and by
$\mathfrak{u}(n)$ its Lie algebra. Let $\lambda=(\lambda_1,
\lambda_2, \ldots, \lambda_n)\in \R^n$ and
$$
\mathcal{H}_\lambda:=\{A\in \mathfrak{u}(n): A^*=-A,\,
\operatorname{spectrum}{A}=i\lambda\}
$$
The unitary group $U(n)$ acts on $\mathcal{H}_\lambda$ by
conjugation. We identify the set of skew-Hermitian matrices
$\mathcal{H}_\lambda$ with a regular coadjoint orbit of $U(n)$ via
the pairing
\begin{align*}
\mathfrak{u}(n)\times \mathfrak{u}(n) &\to \R \\ (X,
Y)&\mapsto\operatorname{Trace}(XY)
\end{align*}
We denote by $\w_\lambda$ the symplectic form obtained by
identifying $\mathcal{H}_\lambda$ with a coadjoint orbit of $U(n).$

Let $T=U(1)^n\subset U(n)$ be the maximal torus of diagonal matrices
in $U(n).$ We identify the Lie algebra of $T$ with $\R^n$ and we
denote by $\{e_1, \cdots, e_n\}$ the standard basis of $\R^n.$ The
system of positive roots associated with the torus $T$ is the set of
vectors
$$
\{\alpha_{i,j}:=e_i-e_j\}_{1\leq i <j\leq n}\subset
\mathfrak{t}\cong \mathfrak{t}^*
$$
The standard system of simple roots is the set
$$
\{\alpha_i:=\alpha_{i,i+1}\}_{1\leq i < n}
$$
The corresponding Dynkin diagram is shown in Figure \ref{A}

\begin{figure}[h!]
\centering
\includegraphics[scale=0.35]{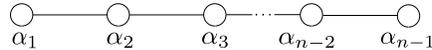}
\caption{Dynkin diagram of $A_n$}\label{A}
\end{figure}
Any $T$-fixed point of $\mathcal{H}_\lambda$ is a permutation of the
diagonal matrix $i(\lambda_1, \ldots, \lambda_n).$ Two $T$-fixed
points of $\mathcal{H}_\lambda$ are joined by one irreducible
$T$-invariant curve if they differ by one transposition.

\begin{teor}\label{HZUn}
Let $\lambda=(\lambda_1, \lambda_2, \ldots, \lambda_n)\in \R^n$ and
assume that $\lambda_1 \geq \lambda_2 \geq \ldots \geq \lambda_n.$
Then,
$$
\operatorname{c_{\operatorname{HZ}}}(\mathcal{H}_\lambda,
\w_\lambda)= \dfrac{1}{2}\sum_{k=1}^n|\lambda_k-\lambda_{n-k+1}|
$$
\end{teor}
\begin{proof}
According to Theorem \ref{upperbound}, in order to find an upper
bound for the Hofer Zehnder capacity of $\mathcal{H}_\lambda,$ we
want to find a path of irreducible $T$-invariant curves joining the
diagonal matrix $i(\lambda_1, \lambda_2, \ldots, \lambda_n)$ with
the diagonal matrix $i(\lambda_n, \lambda_{n-1}, \ldots,
\lambda_1).$ Let us consider the path given by the following
sequence
\begin{align*} i(\lambda_1, \lambda_2, \ldots,
\lambda_{n-1}, \lambda_n) &\xrightarrow{(1,\, n)} i(\lambda_n,
\lambda_2, \ldots, \lambda_{n-1}, \lambda_1)\\ &\xrightarrow{(2,\,
n-1)} \ldots \to i(\lambda_{n}, \lambda_{n-1}, \ldots, \lambda_2,
\lambda_1),
\end{align*}
The degree of this path is equal to
$$
\sum_{k=1}^{[\frac{n-1}{2}]} \check{\alpha}_{k,\,n-k+1}
$$
and its symplectic area is equal to
$$
\w_{\lambda}\Bigl(\sum_{k=1}^{[\frac{n-1}{2}]}
\check{\alpha}_{k,\,n-k+1}\Bigr)=\sum_{k=1}^{[\frac{n-1}{2}]}\langle\lambda,
\check{\alpha}_{k,\,n-k+1}\rangle=
\dfrac{1}{2}\sum_{k=1}^n|\lambda_k-\lambda_{n-k+1}|,
$$ 
and thus
$$
\operatorname{c_{HZ}}(\mathcal{H}_\lambda, \w_\lambda)\leq
\dfrac{1}{2}\sum_{k=1}^n|\lambda_k-\lambda_{n-k+1}|.
$$
Now we show that this inequality is sharp by constructing an
admissible Hamiltonian function $H:\mathcal{H}_\lambda \to \R$ whose
oscillation is equal to the right hand side of the inequality. The
conjugation action of the torus $T$ on $\mathcal{H}_\lambda$ is
Hamiltonian with moment map given by
\begin{align*}
\mu: \mathcal{H}_\lambda &\to i\R^n \\
A &\to \operatorname{diagonal}(A)
\end{align*}
The image of the moment map $\mu$ is the convex hull of all possible
permutations of the vector $i(\lambda_1, \ldots, \lambda_n)\in \R^n$
(see e.g. Guillemin \cite{guillemin}).

For $t\in U(1)$ and $(m_1, \ldots, m_n)\in \Z^n,$ we use the
convention
$$
t^{(m_1, \ldots, m_n)}:=(t^{m_1}, \ldots, t^{m_n})\in
T=U(1)^n\subset U(n).
$$
Let
$$
\beta=\sum_{k=1}^{[\frac{n-1}{2}]}(e_k-e_{n-k+1})
$$
and $S=\{t^\beta=(t, t, \ldots, t^{-1}, t^{-1}):t\in S^1\}\subset
T.$ The action of the circle $S$ on $\mathcal{H}_\lambda$ is
Hamiltonian with moment map given by
\begin{align*}
\tilde{\mu}:\mathcal{H}_\lambda &\to i\R \\
A=(a_{ij}) &\mapsto \langle\mu(A),
\beta\rangle=a_{1,1}-a_{n,n}+a_{2,2}-a_{n-1,n-1}+\ldots
\end{align*}
The moment map image of $\tilde{\mu}$ is the interval
$$
i\Bigl[-\dfrac{1}{2}\sum_{k=1}^n|\lambda_k-\lambda_{n-k+1}|,\dfrac{1}{2}\sum_{k=1}^n|\lambda_k-\lambda_{n-k+1}|\Bigr]\subset
i\R,
$$
and thus the oscillation of $\operatorname{Im}(\tilde{\mu})$ is
equal to $\sum_{k=1}^n|\lambda_k-\lambda_{n-k+1}|.$ Unfortunately,
the function $\operatorname{Im}(\tilde{\mu})$ is not slow. This is
because, under the action of $S$ on $\mathcal{H}_\lambda,$ there are
elements in $\mathcal{H}_\lambda$ with non-trivial finite
stabilizers. All possible stabilizer subgroups of $S$ are either
$\{1\}, \Z_2$ or $S.$ If the stabilizer subgroup of a skew-Hermitian
matrix in $\mathcal{H}_\lambda$ is $\Z_2,$ the period of the orbit
passing through the skew-Hermitian matrix is one half. Otherwise the
skew-Hermitian matrix is either a $S$-fixed point or the period of
the orbit passing through the skew-Hermitian matrix is one.

The Hamiltonian function
$H=\dfrac{1}{2}\operatorname{Im}(\tilde{\mu}):\mathcal{H}_\lambda\to
\R$ fixes this problem. The orbits of $H$ are either constant or
their periods are either one or two. So, the Hamiltonian $H$ is
slow, and
$$
\operatorname{osc}(H)=\dfrac{1}{2}\sum_{k=1}^n|\lambda_k-\lambda_{n-k+1}|
\leq c_{\operatorname{HZ}}(\mathcal{H}_\lambda, \w_\lambda)
$$
and we are done.
\end{proof}
\begin{rk}
The Hofer-Zenhder capacity of the coadjoint orbit
$\mathcal{H}_\lambda$ is the same as the \textit{diameter} of the
weighted Cayley graph of $S_n$ where two permutations are joined by
an edge of weight $|\lambda_i-\lambda_j|$ if they differ by a
trasposition $(i, j).$ In this weighted Cayley graph, the distance
between the identity permutation $e$ and any other permutation
$\sigma$ is given by the expression
$$
d(1, \sigma)=\dfrac{1}{2}\sum_{i=1}^n
|\lambda_i-\lambda_{\sigma(i)}|,
$$
and we have the following rearrangement inequality
$$
0\leq \dfrac{1}{2} \sum_i |\lambda_i-\lambda_{\sigma(i)}| \leq
\dfrac{1}{2} \sum_i|\lambda_i-\lambda_{n-i+1}|
$$
(see e.g. Farnoud and Milenkovic \cite{farnoud}[Theorem 22] for the
first inequality, and Rinott \cite{rinott} and Vince
\cite{vince}[Example 2] for the second).

As an illustrative example, Figure  \ref{bgg24} shows the Bruhat
graph associated with the set of skew-Hermitian matrices
$\mathcal{H}_{(\lambda_1,\lambda_1, 0, 0)}.$ We weight the edges of
the graph with the symplectic areas of the corresponding
$T$-invariant curves. All the symplectic areas are equal to
$|\lambda_1|.$ The diameter of this weighted graph, that is the
Hofer-Zehnder capacity of $\mathcal{H}_{(\lambda_1,\lambda_1, 0,
0)},$ is  equal to $2|\lambda_1|.$

\begin{figure}[h!]
\centering
\includegraphics[scale=0.7]{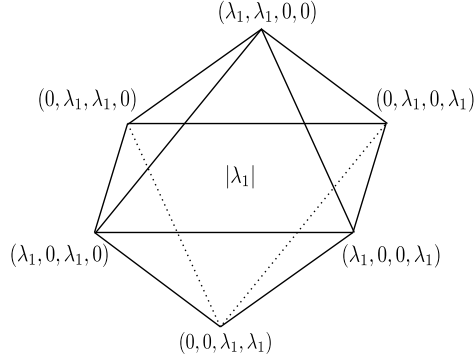}
\caption{Bruhat graph of $\mathcal{H}_{(\lambda_1,\lambda_1, 0,
0)}$}\label{bgg24}
\end{figure}

\end{rk}

\section{Upper bounds for the Hofer-Zehnder capacity of regular coadjoint orbits}

According to the previous section, we can bound from above the
Hofer-Zehnder capacity of a coadjoint orbit of a compact Lie group
by considering paths in its Bruhat graph. In order to achieve an
optimal upper bound for the Hofer-Zehnder capacity through this
method, we want to determine the paths of minimal degree joining the
identity element and the longest element of the Weyl group. A
theorem due to Postnikov states that the minimal degree of such
paths is unique for \textit{regular} coadjoint orbits of compact Lie
groups \cite{postnikov}. Recall that a coadjoint orbit of a compact
Lie group is regular if the stabilizer subgroup of any element in
the coadjoint orbit is a maximal torus.

In this section, we recall Postnikov's results and its combinatorial
formulation in terms of the \textit{quantum Bruhat graph} as it is
done in \cite{postnikov}. We also give a criterion that allow us to
characterize the path of minimal degree joining the identity element
and the longest element in the Bruhat graph.

We follow the same convention as in Section \ref{coadjointorbits}.
Let $G$ be a compact Lie group. Let $T\subset G$ be a maximal torus,
$B\subset G_{\mathbb{C}}$ be a Borel subgroup such that $T\subset B
\subset P.$ Let $R$ and $S$ be the system of roots and simple roots
determined by $T$ and $B,$ respectively.  Let $W$ be the
corresponding Weyl group and $w_0$ be the longest element in $W$
relative to $S.$ For a positive root $\alpha$ write
$$
\check{\alpha}=\sum_{\beta \in
S}\check{n}_{\alpha\beta}\check{\beta}
$$
for some nonnegative integers $\check{n}_{\alpha\beta}$ and define
the height of $\check{\alpha}$ as
$$
\operatorname{ht}(\check{\alpha}):=\sum_{\beta\in
S}\check{n}_{\alpha\beta}
$$
The \textit{quantum Bruhat graph of $W$} is a directed graph on the
elements of the Weyl group with weighted edges defined as follows:
two elements $u, v\in W$ are connected by a directed edge $u
\rightarrow v$ if and only if $v=us_\alpha$ and one of the following
two conditions is satisfied
$$
l(v)=l(u)+1 \ \ \ \text{ or } \ \ \
l(v)=l(u)+1-2\operatorname{ht}(\check{\alpha})
$$
If $l(v)=l(u)+1$ then the degree of the edge equals 0, and if
$l(v)=l(u)+1-2\operatorname{ht}(\check{\alpha})$ then the degree of
the edge equals $\check{\alpha}.$ Note that two vertices can be
connected with edges going in both directions. The \textit{degree}
of a directed path in the quantum Bruhat graph of $W$ is the sum of
degrees of its edges. The \textit{length} of a directed path is the
number of edges that it uses. A directed path in the quantum Bruhat
graph from $u$ to $v$ is \textit{shortest} if it has the minimal
possible length among all direct paths from $u$ to $v.$ Figure
\ref{qbruhat} shows the quantum Bruhat graph of $S_3$ following the
same convention as in Figure \ref{bruhats3}.

\begin{figure}[h!]
\centering
\includegraphics[scale=0.5]{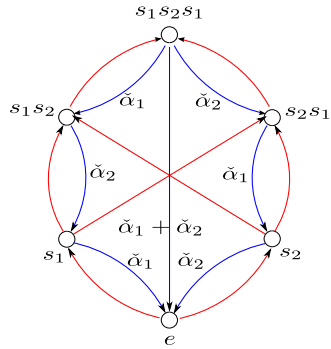}
\caption{Quantum Bruhat graph of $S_3$}\label{qbruhat}
\end{figure}

Now we recall the following result about the combinatorics of paths
in the quantum Bruhat graph of $W.$

\begin{teor}[Postnikov \cite{postnikov}]\label{postnikov}
Let $u$ and $v$ be any two Weyl group elements. There exists a
directed path from $u$ to $v$ in the quantum Bruhat graph. All
shortest paths from $u$ to $v$ have the same degree $d_{\min}(u,
v).$ The degree of any path from $u$ to $v$ is divisible by
$d_{\min}(u, v).$

Moreover, there exists $w\in W$ such that
$$
\GW_{d_{\min}(u, \check{v}), 3}(\sigma_u, \sigma_v, \sigma_w)\ne 0,
$$
and $d_{\min}(u, \check{v})$ is minimal with respect to this
property, i.e., if there exists a degree $d$ and $w\in W$ such that
$$
\GW_{d, 3}(\sigma_u, \sigma_v, \sigma_w)\ne 0
$$
then
$$
d_{\min}(u, \check{v}) \leq d.
$$
\end{teor}
We consider the following lemma whose proof we review for
completeness.
\begin{lemma}
For any positive root $\alpha,$ we always have that
$$
l(s_\alpha)\leq 2\operatorname{ht}(\check{\alpha})-1.
$$
\end{lemma}
\begin{proof}
First notice that for every simple root $\alpha$
$$
\dfrac{1}{2}\sum_{\gamma\in R^+}\langle \gamma, \check{\alpha}
\rangle=1
$$
This is because
$$
\sum_{\gamma\in R^+-\{\alpha\}}\langle \gamma, \check{\alpha}
\rangle=\sum_{\gamma\in
R^+-\{\alpha\}}\dfrac{\gamma-s_{\alpha}(\gamma)}{\alpha}=0
$$
and $\langle \alpha, \check{\alpha}\rangle=2.$ Thus, for any
positive root $\alpha$
$$
\operatorname{ht}(\check{\alpha})=\dfrac{1}{2}\sum_{\gamma\in
R^+}\langle \gamma, \check{\alpha} \rangle
$$
Now, for $w\in W$ let
$$
I(w):=\{\beta\in R^+: w(\beta)<0\}
$$
be the set of \textit{inversions} of $w.$ The second expression uses
the order $\leq$ defined on $\mathfrak{t}^*$ where $a\leq b$ if and
only if $b-a$ is a nonnegative linear combination of simple roots.
Recall that $l(w)=|I(w)|.$ Since $s_\alpha$ stabilizes the set
$R^+-I(s_\alpha)$
$$
\sum_{\gamma\in R^+-I(s_\alpha)}\langle \gamma, \check{\alpha}
\rangle=\sum_{\gamma\in
R^+-I(s_\alpha)}\dfrac{\gamma-s_{\alpha}(\gamma)}{\alpha}=0
$$
For any root $\gamma\in I(s_\alpha)$ we have
$s_\alpha(\gamma)=\gamma-\langle\gamma, \check{\alpha}\rangle\alpha
<0,$ hence $\langle\gamma, \check{\alpha}\rangle\geq 1$ and
$$
2\operatorname{ht}(\check{\alpha})=\sum_{\gamma\in
R^+}\langle\gamma,\check{\alpha}\rangle=\sum_{\gamma\in
I(s_\alpha)}\langle\gamma,\check{\alpha}\rangle=2+\sum_{\gamma\in
I(s_\alpha)-\{\alpha\}}\langle\gamma,\check{\alpha}\rangle\geq
l(s_\alpha)+1.
$$
\end{proof}

Let $l_T:W\to \Z_{\geq 0}$ denote the word length function defined
on $W$ with respect to the generating set of reflections
$\{s_\alpha\}_{\alpha\in R^+},$ i.e., for $w\in W$ if we write
$w=s_{\alpha_1}\cdot\ldots\cdot s_{\alpha_r}$ for some positive
roots $\alpha_1, \cdots, \alpha_r$ and $r$ is minimal, then
$r=l_T(w).$  For $w\in W,$ we call $l_T(w)$ the \textit{absolute
length} of $w.$

The following statement is a consequence of the previous Lemma and
Theorem \ref{postnikov}. It can be used to determine when a
decomposition of $w_0$ into a product of reflections gives rise to a
shortest path joining $e$ and $w_0.$ We provide in the statement the
corresponding upper bound for the Hofer-Zehnder capacity.

\begin{teor}\label{Ca3}
Let $w_0$ be the longest element in the Weyl group $W$ with respect
to the system of simple roots $S.$ If there exist positive roots
$\alpha_1, \cdots, \alpha_r$ such that $w_0=s_{\alpha_1}\cdot \ldots
\cdot s_{\alpha_r}$ with $r=l_T(w_0)$ and
$$
\sum_{i=1}^r(2\operatorname{ht}(\check{\alpha}_i)-1)=l(w_0)=|R^+|,
$$
then
$$
d_{\min}(w_0, e)=\sum_{i=1}^r \check{\alpha_i}
$$
In addition, if $\lambda$ is in the interior of the Weyl chamber
relative to the set of simple roots $S,$
$$
\min_d\, \w_\lambda(d)=\sum_{i=1}^r \langle \lambda,
\check{\alpha_i} \rangle
$$
where the minimum is taken over all degrees of \textit{paths}
joining $e$ with $w_0$ in the standard Bruhat graph of $W.$ In
particular we obtain the following upper bound for the Hofer-Zehnder
capacity of $(\mathcal{O}_\lambda, \w_\lambda)$
$$
c_{\operatorname{HZ}}(\mathcal{O}_\lambda, \w_\lambda) \leq
\sum_{i=1}^r \langle \lambda, \check{\alpha_i} \rangle
$$
\end{teor}
\begin{proof}
For every $1\leq i < r,$ we have that
\begin{align*}
l(w_0)&\leq l(s_{\alpha_1}\cdot \ldots \cdot
s_{\alpha_{i}})+l(s_{\alpha_{i+1}}\cdot \ldots \cdot s_{\alpha_{r}})
\leq \sum_{j=1}^il(s_{\alpha_j})+\sum_{j=i+1}^rl(s_{\alpha_j})\\
&\leq
\sum_{j=1}^i(2\operatorname{ht}(\check{\alpha}_j)-1)+\sum_{j=i+1}^r(2\operatorname{ht}(\check{\alpha}_j)-1)=l(w_0)
\end{align*}
Thus for every $1\leq i < r,$
$$
l(s_{\alpha_1}\cdot \ldots \cdot
s_{\alpha_{j}})=\sum_{i=1}^j(2\operatorname{ht}(\check{\alpha}_i)-1)
$$
In particular
$$
l(s_{\alpha_1}\cdot \ldots \cdot s_{\alpha_{i}})=l(s_{\alpha_1}\cdot
\ldots \cdot s_{\alpha_{i}}\cdot s_{\alpha_{i+1}})
+1-2\operatorname{ht}(\check{\alpha}_{i+1})
$$
and
$$
w_0=s_1 \cdot s_2\cdot \ldots \cdot s_r \to \cdots \to s_1 \cdot
s_2\to s_1 \to e
$$
represents a directed path in the quantum Bruhat graph from $w_0$ to
$e.$ This path is a shortest path because of the minimality property
of $r=l_T(w_0),$ and hence its degree equals $d_{\min}(w_0, e).$
\end{proof}
\begin{rk}
In section \ref{computation}, we verify that for the longest element
$w_0$ in the Weyl group $W$ there exist positive roots $\alpha_1,
\cdots, \alpha_r$ such that $w_0=s_{\alpha_1}\cdot \ldots \cdot
s_{\alpha_r}$ with $r=l_T(w_0).$ Hence the assumptions made in the
last theorem hold for any compact Lie group.
\end{rk}

\section{Lower bounds for the Hofer-Zehnder capacity and Hamiltonian torus actions}

In this section we describe how to compute lower bounds for the
Hofer-Zehnder capacity of a symplectic manifold with a Hamiltonian
torus action by using its moment map. By simplicity, we always
assume that the points fixed by the torus action are isolated. We
also estimate from below the Hofer-Zenhder capacity of coadjoint
orbits of compact Lie groups as it was already done in the proof of
Theorem \ref{HZUn} for coadjoint orbits of the unitary group.

Let $T$ denote a torus. In what follows we always identify the Lie
algebra $\mathfrak{u}(1)$ with $\R.$ A \textit{weight} of $T$ is a
Lie group morphism $\eta:T\to U(1).$ A \textit{coweight} of $T$ is a
Lie group morphism $\xi:U(1)\to T.$ Let $\Lambda \subset
\mathfrak{t}$ be the kernel of the exponential map $\exp:
\mathfrak{t} \to T$ and $\Lambda^*=\operatorname{Hom}(\Lambda, \Z)$
be its dual. The differential of any weight $\mu:T\to U(1)$ is a Lie
algebra morphism $\mathfrak{t}\to \mathfrak{u}(1)\cong \R$ that
takes $\Lambda$ into $2\pi \Z.$ Conversely, any group morphism
$\Lambda \to 2\pi \Z$ arises in this way. Thus, the set of weights
$X^*(T):=\operatorname{Hom}(T, U(1))$ can be identified with $2\pi
\Lambda^*\subset \mathfrak{t}^*.$ Similarly, the set of coweights
$X_*(T):=\operatorname{Hom}(U(1), T)$ is identified with
$\dfrac{1}{2\pi}\Lambda\subset \mathfrak{t}.$ In this section, we
always see weights and coweights as elements of $\mathfrak{t}^*$ and
$\mathfrak{t},$ respectively. When we pair a coweight $\xi$ with a
weight $\eta,$ we denote the composition $\eta\circ \xi$ by $\langle
\eta, \xi \rangle.$ We always identify $\operatorname{Hom}(U(1),
U(1))$ with $\Z.$

The Schur Lemma implies that for any representation $V $of $T,$ we
can write $V$ as a direct sum
$$
V=\bigoplus_{\eta\in X^*(T)}V_\eta
$$
where $V_\eta=\{v\in V: t\cdot v=\eta(t)v \text{ for all }t\in T
\}.$ We call a $\eta$ such that $V_\eta\ne \{0\}$ a weight of the
representation. Any coweight  $\nu$ of $T$ defines a representation
of $U(1)$ on $V$ by pulling back the action of $T$ on $V$ and the
weights of this representation are the nonzero elements of the set
of integers $\{\langle \eta, \nu \rangle:\eta \in  X^*(T)\}.$

Let $(M, \w)$ be a symplectic manifold with a Hamiltonian action of
a torus $T$ generated by a moment map $\phi:M\to \mathfrak{t}^*.$
The critical points of $\phi$ are the fixed points of the torus
action. In this section, we assume that the number of fixed points
is finite.

For every fixed point $p\in M,$ the \textit{isotropy weights at $p$}
are the weights $\eta_1, \cdots, \eta_n$ such that the tangent space
$T_pM$ is linearly symplectomorphic to the action on $(\C^n,
\w_{\operatorname{st}})$ defined by
$$
t\cdot (z_1, \cdots, z_n):=(\eta_1(t)z_1,\cdots, \eta_n(t)z_n)
$$
and generated by the moment map
\begin{align*}
\C^n &\to\mathfrak{t}^* \\
(z_1, \cdots, z_n) &\mapsto \dfrac{1}{2}(|z_1|^2\eta_1+ \cdots +
|z_n|^2\eta_n ).
\end{align*}
An \textit{isotropy weight of the torus action} is an isotropy
weight of the torus action at some fixed point.

For $\xi \in \mathfrak{t},$ define
\begin{align*}
\phi^{\xi}:M&\to \R \\
p &\mapsto \langle \phi(p), \xi \rangle
\end{align*}
We call $\xi \in \mathfrak{t}$ \textit{generic} if $\langle\eta,
\xi\rangle\ne 0 $ for every isotropy weight $\eta$ of the torus
action. In this case, the function $\phi^\xi:M\to \R$ is Morse and
its set of critical points coincides with the set of points fixed by
the torus action.

Let $\xi$ be a coweight of $T.$ The coweight $\xi$ defines a
Hamiltonian circle action on $M$ with moment map $\phi^\xi:M \to
\R.$ If the isotropy weights of the torus action at a fixed point
are $\eta_1, \cdots, \eta_n,$ the isotropy weights of the circle
action defined by the coweight $\xi$ are the integers $\langle
\eta_1, \xi \rangle, \cdots, \langle \eta_n, \xi\rangle.$ The
coweight $\xi$ is generic if for every fixed point all the integers
$\langle \eta_1, \xi \rangle, \cdots, \langle \eta_n, \xi\rangle$
are nonzero.

\begin{teor}\label{Ca4}
Let $(M^{2n}, \w)$ be a compact symplectic manifold with a
Hamiltonian circle action $S^1$ generated by a moment map $H:M\to
\R.$ Assume that the number of fixed points is finite. Let $I
\subset \Z$ be the set of all isotropy weights of the circle action
and
$$
m^+=\max_{m\in I} |m|.
$$
Then the function
$$
H'=\frac{1}{m^+}H:M\to \R
$$
is slow and in particular
$$
\operatorname{osc}{H'}\leq c_{\operatorname{HZ}}(M, \w)
$$
\end{teor}
\begin{proof}
We show that the stabilizer subgroup of $S^1$ at every non fixed
point is a cyclic subgroup of order less or equal to $m^+.$ This
claim implies the theorem.

Let $p$ be a fixed point of $S^1$ and $m_1, \cdots, m_n$ be the
isotropy weights at $p.$ The equivariant Darboux theorem asserts
that there is a neighbourhood $U$ of $p$ in $M$ equivariantly
symplectomorphic to a neighborhood $V$ of the origin in $(\C^n,
\w_{\operatorname{st}})$ with the circle action defined by
$$
t\cdot (z_1, \cdots, z_n)=(t^{m_1}z_1,\cdots, t^{m_n}z_n)
$$
The stabilizer subgroup of $(z_1, \cdots, z_n)\in V \backslash\{0\}$
is a cyclic group of order equal to $\gcd\{m_i:z_i\ne 0\}.$ Note
that $\gcd\{m_i:z_i\ne 0\}$ is less or equal to $m^+,$ and the claim
holds at any point of the Darboux chart.

Finally, the stabilizer group of a point located anywhere in the
manifold coincides with the stabilizer group of some point located
at some equivariant Darboux's chart of some fixed point (see e.g.
Guillemin, Lerman and Sternberg \cite{sympf}[Lemma 3.3.2]). The
statement follows from the analysis  done in the previous
paragraphs.
\end{proof}

\begin{corol}\label{cor}
Let $(M^{2n}, \w)$ be a compact symplectic manifold with a
Hamiltonian torus action $T$ generated by a moment map $H:M\to
\mathfrak{t}^*.$ Let us assume that the set of fixed points of $T$
are isolated. Let $I\subset X^*(T)$ be the set of isotropy weights
of the torus action. For a coweight $\xi \in X_*(T),$ define
$$
m_\xi^+:=\max_{\eta\in I}|\langle \eta, \xi\rangle |
$$
Then
$$
\sup\Bigl\{\dfrac{1}{m_\xi^+}
\operatorname{osc}(\phi^{\xi}):\text{generic $\xi\in X_*(T)$}\Bigr\}
\leq c_{\operatorname{HZ}}(M, \w)
$$
\end{corol}

We want to apply the previous Corollary to bound from below the
Hofer-Zehnder capacity of coadjoint orbits of compact Lie groups. We
use the same convention as in Section \ref{coadjointorbits}. Let $G$
be a compact \textit{simple} Lie group, $T\subset G$ be a maximal
torus and $W$ be the corresponding Weyl group. Let $R$ be the
corresponding system of roots relative to $T$ and $S$ be a choice of
simple roots. For any positive root $\alpha,$ we write
$$
\alpha=\sum_{\beta\in S}n_{\alpha\beta}\beta,
$$
for some nonnegative integers $n_{\alpha\beta}.$ We identify
$\mathfrak{g}$ and $\mathfrak{g}^*$ via an adjoint invariant inner
product $(\cdot\, , \cdot).$ Let $\lambda\in \mathfrak{t}^*_+$ be an
element of the Weyl chamber relative to $S$ and
$\mathcal{O}_\lambda$ be the coadjoint orbit passing through
$\lambda$. The maximal group $T$ acts hamiltonially on
$\mathcal{O}_\lambda $ with moment map $\phi:\mathcal{O}_\lambda
\hookrightarrow \mathfrak{t}^*$ equals to the composition of the
projection map $\mathfrak{g}^*\to \mathfrak{t}^*$ with the inclusion
map $\mathcal{O}_\lambda \hookrightarrow \mathfrak{g}^*.$ The image
of $\phi$ is the convex hull of $\{w(\lambda)\}_{w\in W}.$ The set
of all isotropy weights of the torus action of $T$ on
$\mathcal{O}_\lambda$ is a subset of the set of roots and equal to
the whole set of roots when $\mathcal{O}_\lambda$ is a regular
coadjoint orbit.

We say that $\alpha \leq \beta$ for two positive roots $\alpha$ and
$\beta,$ if $\beta-\alpha$ is a nonnegative linear combination of
simple roots. The \textit{highest positive root} is the positive
root that is maximal with respect to the order that we define for
positive roots. The existence and uniqueness of the highest positive
root follows from the fact that $G$ is simple.

In the next statement we give our lower bound for the Hofer-Zehnder
capacity of coadjoint orbits. In the proof, we keep the notation of
Corollary \ref{cor}.

\begin{teor}\label{Ca2}
Let $\rho$ be the \textit{highest positive root}. Assume that the
longest element $w_0$ in $W$ relative to $S$ can be decomposed as
$$
w_0=s_{\alpha_1}\cdot \ldots \cdot s_{\alpha_r}
$$
where $\alpha_1, \cdots, \alpha_r$ are pairwise orthogonal positive
roots. Then,
\[
\max_{\alpha \in S}\Bigl\{\sum_{k=1}^r
\dfrac{n_{\alpha_k\,\alpha}}{n_{\rho\,\alpha}} \langle \lambda,
\check{\alpha}_k \rangle \Bigr\} \leq
c_{\operatorname{HZ}}(\mathcal{O}_\lambda, \w_\lambda)
\]
\end{teor}
\begin{proof}

For $w\in W,$ the weight decomposition of the tangent space of
$\mathcal{O}_\lambda$ at $w(\lambda)$ is
$$
T_{w(\lambda)}\mathcal{O}_\lambda=\bigoplus_{\alpha\in
R^+-R^+_P}\mathfrak{g}_{-w(\alpha)}
$$
The isotropy weights of the circle action defined by a coweight
$\xi$ is the set of integers
$$
\{-\langle \alpha, \xi \rangle :\alpha\in W(R^+-R^+_P)\}.
$$
We call a coweight $\xi$ positive if
$$
\langle \alpha, \xi\rangle > 0
$$
for every positive root $\alpha.$ We denote the set of positive
coweights by $X_*(T)_+.$ Every positive coweight $\xi$ is generic by
definition. We can always assume that a coweight is positive by
taking a different system of simple roots if needed.

For a positive coweight $\xi,$ the Morse index of
$\phi^\xi:\mathcal{O}_\lambda \to \R$ at $\lambda$ is the maximum
possible and $\lambda$ is a local maximum. Indeed, $\lambda$ is an
absolute maximum. Similarly, the absolute minimum of
$\phi^\xi:\mathcal{O}_\lambda \to \R$ is achieved at $w_0(\lambda).$
Note that the value
$$
m_\xi^+=\max_{\alpha \in W(R-R_P)}|\langle  \alpha, \xi \rangle|
$$
is achieved when $\alpha$ is the highest positive root $\rho.$

Our orthogonality assumption implies that
$$
w_0(\lambda)= \prod_{k=1}^r
s_{\alpha_k}(\lambda)=\lambda-\sum_{k=1}^r \langle \lambda,
\check{\alpha}_k \rangle \alpha_k
$$
and for a positive coweight $\xi$
$$
\operatorname{osc}(\phi^\xi)=\sum_{k=1}^r \langle \lambda,
\check{\alpha}_k \rangle \langle \alpha_k, \xi \rangle
$$
Following Corollary \ref{cor}, we want to maximize the expression
$$
\dfrac{1}{\langle \rho, \xi
\rangle}\operatorname{osc}(\phi^\xi)=\sum_{k=1}^r \dfrac{\langle
\alpha_k, \xi \rangle}{\langle \rho, \xi \rangle}\langle \lambda,
\check{\alpha}_k \rangle
$$
for $\xi \in X_*(T)_+.$ The right hand side of the last equation is
scale invariant and continuous as a function of the variable $\xi$
on the convex cone
$$
X_*(T)_+\otimes \R=\{\xi\in \mathfrak{t}\backslash\{0\}:\langle
\alpha, \xi \rangle\geq 0 \text{ for all }\alpha\in R^+\}
$$
Thus,
$$
\sup_{\xi\in X_*( T)_+}\Bigl\{\sum_{k=1}^r \dfrac{\langle \alpha_k,
\xi \rangle}{\langle \rho, \xi \rangle} \langle \lambda,
\check{\alpha}_k \rangle \Bigr\}=\sup_{\xi\in
 X_*(T)_+\otimes\mathbb{R}}\Bigl\{\sum_{k=1}^r \dfrac{\langle \alpha_k, \xi \rangle}{\langle
\rho, \xi \rangle}\langle \lambda, \check{\alpha}_k \rangle \Bigr\}
$$
The change of variable
$$
y:=\dfrac{\xi}{\langle \rho, \xi \rangle}
$$
transform our problem into the following linear optimization problem
$$
\begin{cases}
\text{Maximize } & \sum_{k=1}^r \langle \alpha_k, y \rangle\langle
\lambda, \check{\alpha}_k \rangle
\\
\text{Subject to } & \langle \rho, y \rangle =1 \\ & \langle \alpha,
y \rangle \geq  0 \text{  for all }\alpha \in S
\end{cases}
$$
The hyperplane in $\mathfrak{t}$ defined by the equation $\langle
\rho, y \rangle =1$ cuts $X_*(T)_+\otimes \R$ into the polytope
$$
\triangle=\{y\in \mathfrak{t}: \langle \rho, y \rangle =1,
\,\,\langle \alpha, y \rangle \geq 0 \text{  for all }\alpha \in S
\}.
$$
The maximum value of the linear expression $\sum_{k=1}^r \langle
\alpha_k, y \rangle\langle \lambda, \check{\alpha}_k \rangle$ on
$\triangle$ is obtained at some of the vertices of the polytope
$\triangle.$ Equivalently, the maximum value of the expression
$$
\sum_{k=1}^r  \dfrac{\langle \alpha_k, \xi \rangle}{\langle \rho,
\xi \rangle}\langle \lambda, \check{\alpha}_k \rangle
$$
is obtained at some of the one-dimensional faces of $X_*(T)_+\otimes
\R.$ Each one-dimensional face of $X_*(T)_+\otimes \R$ is spanned by
some element in the basis dual to the basis of simple roots defined
by the relation
$$
(\tau_\alpha, \beta)=\delta_{\alpha, \beta} \text{ for any }\alpha,
\beta\in S.
$$
We conclude that
\[
\sup_{\xi\in X_*(T)_+\otimes \mathbb{R}}\Bigl\{\sum_{k=1}^r
\dfrac{\langle \alpha_k, \xi \rangle}{\langle \rho, \xi
\rangle}\langle \lambda, \check{\alpha}_k \rangle \Bigr\}=
\max_{\alpha \in S}\Bigl\{ \sum_{k=1}^r \dfrac{(\alpha_k,
\tau_\alpha)}{(\rho, \tau_\alpha)}\langle \lambda, \check{\alpha}_k
\rangle  \Bigr\}
\]
By Corollary \ref{cor}, we get that
\[
\max_{\alpha \in S}\Bigl\{ \sum_{k=1}^r \dfrac{(\alpha_k,
\tau_\alpha)}{(\rho, \tau_\alpha)}\langle \lambda, \check{\alpha}_k
\rangle  \Bigr\}=\max_{\alpha \in S}\Bigl\{\sum_{k=1}^r
\dfrac{n_{\alpha_k\,\alpha}}{n_{\rho\,\alpha}}\langle \lambda,
\check{\alpha}_k \rangle\Bigr\} \leq
c_{\operatorname{HZ}}(\mathcal{O}_\lambda, \w_\lambda)
\]
\end{proof}

\begin{rk}
The previous statement is compatible with Theorem \ref{Ca}, i.e.,
with the same notation as in the previous theorem, we have that
\[
\max_{\alpha \in S}\Bigl\{\sum_{k=1}^r
\dfrac{n_{\alpha_k\,\alpha}}{n_{\rho\,\alpha}}\langle \lambda,
\check{\alpha}_k \rangle \Bigr\} \leq \sum_{k=1}^r \langle \lambda,
\check{\alpha}_k \rangle
\]
More generally, let $(M, \w)$ be a symplectic manifold and assume
that $S^1$ acts Hamiltonially on $(M, \w)$ with a moment map $H:M\to
\R.$ Assume that the fixed points of the circle action are isolated.

Let $p_0, p_1, \cdots, p_n$ be a sequence of fixed points such that
every pair $p_i, p_{i+1}$ of critical points in the sequence is
joined by a $S^1$-invariant sphere $S_{i}$ and $p_0, p_n$ are the
minimum and maximum of $H,$ respectively.

\begin{figure}[h!]
\centering
\includegraphics[scale=0.7]{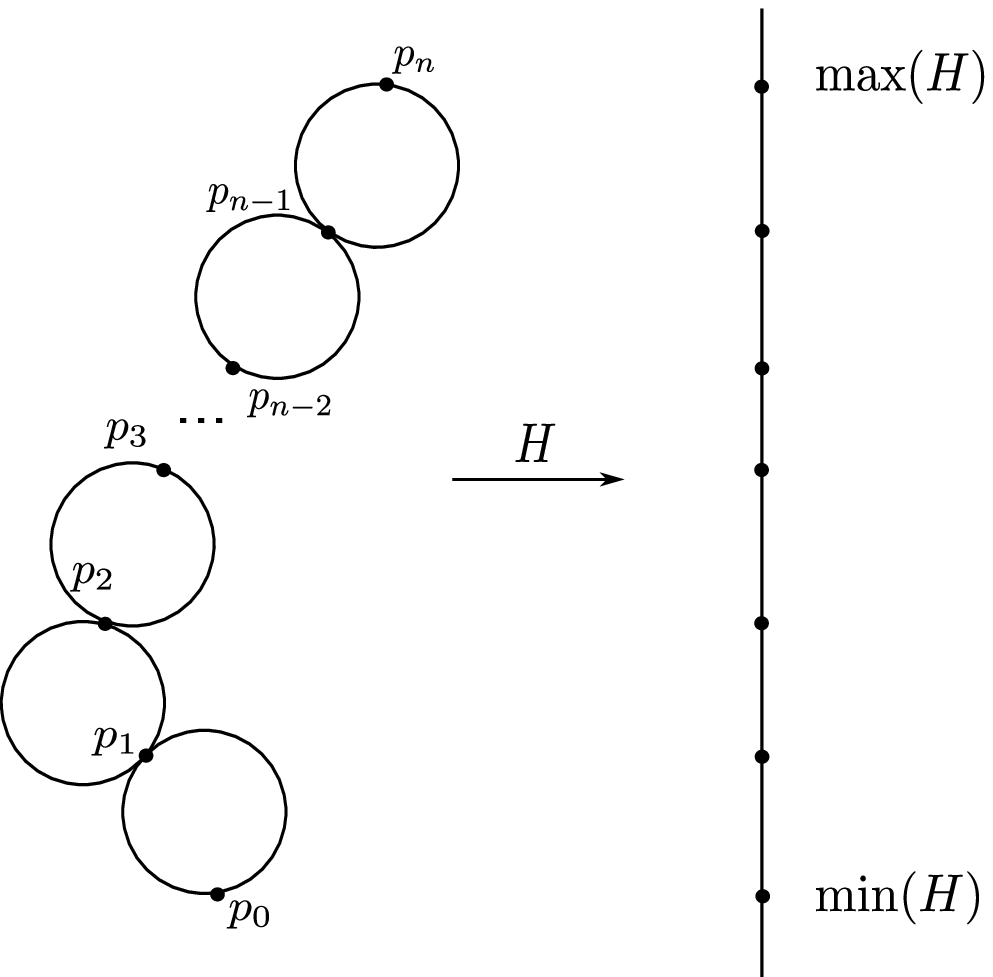}
\end{figure}

If the isotropy weight of the circle action restricted to the sphere
$S_i$  at $p_i$ is $k_i,$ then
$$
|H(p_{i+1})-H(p_i)|=k_i\w(S_i)
$$
(see e.g. McDuff and Tolman \cite{McDuffTolman}[Lemma 3.9]). Thus
\begin{align*}
\operatorname{osc}{H}&=H(p_n)-H(p_0)\leq \sum_i |H(p_{i+1})-H(p_i)|
= \sum_i k_i\w(S_i)\\
&\leq \max_i{k_i}\sum_i\w(S_i)\leq m^+\sum_i\w(S_i),
\end{align*}
and
$$
\frac{1}{m^+}\operatorname{osc}{H} \leq \sum_i\w(S_i),
$$
where $m^+$ is defined as in Theorem \ref{Ca4}. When the symplectic
manifold is a  coadjoint orbit and the circle action comes from a
coweight of a maximal torus, the Hofer-Zehnder capacity of the
coadjoint orbit is between the two values of the last inequality.
\end{rk}

\section{Computation of bounds Hofer-Zehnder
capacity}\label{computation}

In this section we show that the assumptions made in Theorem
\ref{Ca3} and Theorem \ref{Ca2} hold for any Weyl group and we
compute for any compact simple Lie groups the corresponding bounds
for the Hofer-Zehnder capacity of their regular coadjoint orbits.

We use the same convention as in Section \ref{coadjointorbits}. Let
$G$ be a compact \textit{simple} Lie group and $T\subset G$ be a
maximal torus. Let $R$ be the set of roots associated with $T$ and
$S$ be a choice of simple roots. We denote by $W$ the corresponding
Weyl group.

In the following theorem, we summarize the main results of the
paper.
\begin{teor}
Let $w_0$ be the longest element of $W$ relative to the set of
simple roots $S.$ There exist pairwise orthogonal positive roots
$\alpha_1, \cdots, \alpha_r$ such that $l_T(w_0)=r,$
$$
w_0=s_{\alpha_1}\cdot \ldots \cdot s_{\alpha_r}
$$
and
$$
\sum_{i=1}^r(2\operatorname{ht}(\check{\alpha}_i)-1)=l(w_0)=|R^+|
$$
In particular, for regular $\lambda\in \mathfrak{t}^*_+$ we obtain
the following bounds for the Hofer-Zehnder capacity of the coadjoint
orbit  $\mathcal{O}_\lambda$ with respect to its
Kostant-Kirillov-Souriau form $\w_\lambda$
\[
\max_{\alpha \in S}\Bigl\{\sum_{k=1}^r
\dfrac{n_{\alpha_k\,\alpha}}{n_{\rho\,\alpha}}\langle \lambda,
\check{\alpha}_k \rangle  \Bigr\} \leq
c_{\operatorname{HZ}}(\mathcal{O}_\lambda, \w_\lambda) \leq
\sum_{k=1}^r \langle \lambda, \check{\alpha}_k \rangle,
\]
where $\rho$ denotes the highest positive root.
\end{teor}
We split the proof of the previous statement in several cases
according to the type of the Lie group $G.$ We provide the described
decomposition of $w_0$ and the corresponding lower and upper bound
for the Hofer-Zehnder capacity of the regular coadjoint orbit
$\mathcal{O}_\lambda.$ We omit the detailed calculations, although
we give enough information so they can be verified by the reader.

Let $\lambda$ be in the interior of the Weyl chamber relative to
$S,$ $\mathcal{O}_\lambda$ be the coadjoint orbit passing through
$\lambda$ and $\w_\lambda$ be the Kostant-Kirillov-Souriau form
defined on $\mathcal{O}_\lambda.$

\subsection*{Type $B$}

The standard root system for the group $B_n=SO(2n+1)$ is identified
with the set of vectors $R=\{\pm e_i,\, \pm(e_j\pm e_k):\, j\ne
k\}_{1\leq i, j \leq n}\subset \R^n$ with a choice of simple roots
given by $S=\{\alpha_1=e_1-e_2, \cdots, \alpha_{n-1}=e_{n-1}-e_n,
\alpha_n=e_n\}.$ The Dynkin diagram of $B_n$ is shown in Figure
\ref{B}

\begin{figure}[h!]
\centering
\includegraphics[scale=0.7]{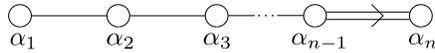}
\caption{Dynkin diagram of $B_n$}\label{B}
\end{figure}

The longest element $w_0$ of $B_n$ seen as a map of $\R^n$ is the
reflection
\begin{align*}
\R^n &\to \R^n \\
(x_1, \cdots, x_n)&\mapsto (-x_1, \cdots, -x_n)
\end{align*}
We have that $l_T(w_0)=n, l(w_0)=n^2$ and
$$
w_0=\begin{cases} s_{e_1-e_2}s_{e_1+e_2}
s_{e_3-e_4}s_{e_3+e_4}\ldots \cdot
s_{e_{n-1}-e_n}s_{e_{n-1}+e_n} &\text{ if $n$ is even}\\
s_{e_1-e_2}s_{e_1+e_2}s_{e_3-e_4}s_{e_3+e_4}\cdot \ldots \cdot
s_{e_{n-2}-e_{n-1}}s_{e_{n-2}+e_{n-1}}s_{e_n} &\text{ if $n$ is odd}
\end{cases}
$$
Hence
$$
c_{\operatorname{HZ}}(\mathcal{O}_\lambda, \w_\lambda) \leq
\begin{cases}
2\lambda_1+2\lambda_3+\cdots+2\lambda_{n-1} &\text{ if $n$ is
even}\\
2\lambda_1+2\lambda_3+\cdots+2\lambda_{n-2}+2\lambda_{n} &\text{ if
$n$ is odd}\\
\end{cases}
$$
The highest root $\rho$ is $e_1+e_2$ and
$$
c_{\operatorname{HZ}}(\mathcal{O}_\lambda, \w_\lambda) \geq
\max\{2\lambda_1, \lambda_1+\lambda_2+\cdots+\lambda_n\}
$$

\subsection*{Type $C$}

The standard root system for the group $C_n=Sp(n)$ is identified
with the set of vectors $R=\{\pm 2e_i,\, \pm(e_j\pm e_k):\, j\ne
k\}_{1\leq i, j \leq n}\subset \R^n$ with a choice of simple roots
given by $S=\{\alpha_1=e_1-e_2, \cdots, \alpha_{n-1}=e_{n-1}-e_n,
\alpha_n=2e_n\}.$ The Dynkin diagram of $C_n$ is shown in Figure
\ref{C}

\begin{figure}[h!]
\centering
\includegraphics[scale=0.7]{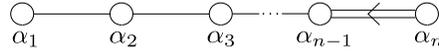}
\caption{Dynkin diagram of $C_n$}\label{C}
\end{figure}

Combinatorially speaking, the Weyl group $C_n$ is the same as the
Weyl group  $B_n,$ however the edges of its Bruhat graphs have
different degrees.

As an automorphism of $\R^n,$ the longest element $w_0$ of $C_n$ is
the reflection
\begin{align*}
\R^n &\to \R^n \\
(x_1, \cdots, x_n)&\mapsto (-x_1, \cdots, -x_n)
\end{align*}
We have $l_T(w_0)=n, l(w_0)=n^2$ and
$$
w_0=s_{2e_1}s_{2e_2}\cdot \ldots \cdot s_{2e_n}
$$
Hence,
$$
c_{\operatorname{HZ}}(\mathcal{O}_\lambda, \w_\lambda) \leq
\lambda_1+\lambda_2+\cdots+\lambda_n
$$
The longest root is $\rho=2e_1$ and
$$c_{\operatorname{HZ}}(\mathcal{O}_\lambda, \w_\lambda) \geq \lambda_1+\lambda_2+\cdots+\lambda_n
$$
Hence, we get the sharp expression for coadjoint orbits of type $C$
$$
c_{\operatorname{HZ}}(\mathcal{O}_\lambda, \w_\lambda)=
\lambda_1+\lambda_2+\cdots+\lambda_n
$$

\subsection*{Type $D$}

The standard root system for the group $D_n=SO(2n)$ is identified
with the set of vectors $R=\{\pm(e_j\pm e_k):\, j\ne k\}_{1\leq i, j
\leq n}\subset \R^n$ with a choice of simple roots given by
$S=\{\alpha_1=e_1-e_2, \cdots, \alpha_{n-1}=e_{n-1}-e_n,
\alpha_n=e_{n-1}+e_n\}.$ The Dynkin diagram of $D_n$ is shown in
Figure \ref{D}.

\begin{figure}[h!]
\centering
\includegraphics[scale=0.7]{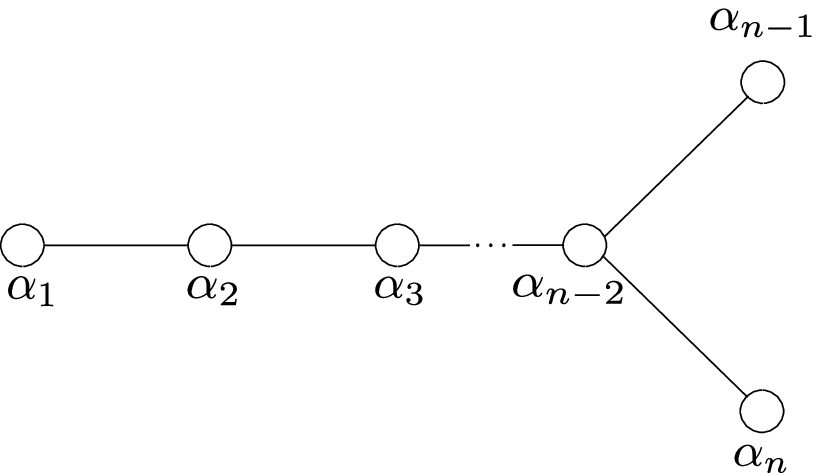}
\caption{Dynkin diagram of $D_n$}\label{D}
\end{figure}

As a map of $\R^n,$ the longest element $w_0$ is the application
\begin{align*}
\R^n &\to \R^n \\
(x_1, \cdots, x_{n-1}, x_n)&\mapsto
\begin{cases}
(-x_1, \cdots, -x_{n-1}, -x_n) & \text{if $n$ is
even}\\
(-x_1, \cdots, -x_{n-1}, x_n) & \text{if $n$ is odd}
\end{cases}
\end{align*}
We have that
$$
l_T(w_0)=\begin{cases} n & \text{if $n$ is even}\\
n-1 & \text{if $n$ is odd}
\end{cases},
$$
$l(w_0)=n(n-1),$ and
$$
w_0=\begin{cases} s_{e_1-e_2}s_{e_1+e_2}s_{e_3-e_4}s_{e_3+e_4}\cdot
\ldots \cdot
s_{e_{n-1}-e_n}s_{e_{n-1}+e_n} &\text{ if $n$ is even}\\
s_{e_1-e_2}s_{e_1+e_2}s_{e_3-e_4}s_{e_3+e_4}\cdot \ldots \cdot
s_{e_{n-2}-e_{n-1}}s_{e_{n-2}+e_{n-1}} &\text{ if $n$ is odd}
\end{cases}
$$
Hence,
$$
c_{\operatorname{HZ}}(\mathcal{O}_\lambda, \w_\lambda)\leq
\begin{cases}2\lambda_1+ 2\lambda_3+ \ldots +2\lambda_{n-1} &\text{ if $n$ is even}\\
2\lambda_1+ 2\lambda_3+\ldots +2\lambda_{n-2}  &\text{ if $n$ is
odd}
\end{cases}
$$
On the other hand, $\rho=e_1+e_2$ and
$$
c_{\operatorname{HZ}}(\mathcal{O}_\lambda, \w_\lambda)\geq
\begin{cases}\max\{2\lambda_1, \lambda_1+\lambda_2+\cdots+\lambda_{n-1}+|\lambda_{n}|\}  &\text{ if $n$ is even}
\\ \max\{2\lambda_1, \lambda_1+\lambda_2+\cdots+\lambda_{n-1}\}
 &\text{ if $n$ is odd}
\end{cases}
$$

\subsection*{Type $E$}

There are three isomorphism classes of compact simple Lie groups of
type $E: E_6, E_7, E_8.$ We start first with $E_8.$ A system of
simple roots for $E_8$ as vectors in $\R^8$ is
\begin{align*}
S=\{&\alpha_1=\dfrac{1}{2}(e_1-e_2-e_3-e_4-e_5-e_6-e_7+e_8),
\alpha_2=e_1+e_2, \alpha_3=-e_1+e_2, \\ &\alpha_4=-e_2+e_3,
\alpha_5=-e_3+e_4, \alpha_6=-e_4+e_5, \alpha_7=-e_5+e_6,
\alpha_8=-e_6+e_7\}
\end{align*}
The Dynkin diagram of $E_8$ is shown in Figure \ref{E}.

\begin{figure}[h!]
\centering
\includegraphics[scale=0.7]{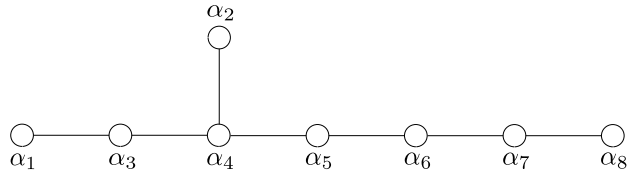}
\caption{Dynkin diagram of $E_8$}\label{E}
\end{figure}

As a map of $\R^8,$ the longest element of $E_8$ is the application
\begin{align*}
\R^8 &\to \R^8 \\
(x_1, \cdots, x_8) &\mapsto (-x_1, \cdots, -x_8)
\end{align*}
and its absolute length and length are equal to 8 and 120,
respectively. We can write the longest element as the composition of
reflections $s_{r_1}, \ldots s_{r_7}$ and $s_{r_8}$ where
\begin{align*}
&r_1=-e_1+e_2,\, r_2=e_1+e_2,\, r_3=-e_3+e_4,\, r_4=e_3+e_4,\\
&r_5=-e_5+e_6,\, r_6=e_5+e_6,\, r_7=-e_7+e_8,\, r_8=e_7+e_8
\end{align*}
The upper bound for the Hofer-Zehnder capacity of a regular
coadjoint orbit $(\mathcal{O}_\lambda, \w_\lambda)$ of $E_8$ is
given by
$$
c_{\operatorname{HZ}}(\mathcal{O}_\lambda, \w_\lambda) \leq
2\lambda_2+2\lambda_4+2\lambda_6+2\lambda_8
$$
The highest root equals to $e_7+e_8$ and
$$
c_{\operatorname{HZ}}(\mathcal{O}_\lambda, \w_\lambda) \geq
\max\Bigl\{2\lambda_8,
\dfrac{1}{3}(\lambda_1+\lambda_2+\cdots+\lambda_7+5\lambda_8)\Bigr\}
$$
We have finished our analysis for $E_8$ and now we continue with the
one for $E_7.$ We keep the notation used in the previous paragraphs.
A system of simple roots for $E_7$ is the set $\{\alpha_1, \alpha_2,
\cdots, \alpha_7\}.$ Note that the Dynkin diagram of $E_7$ is
contained in the Dynkin diagram of $E_8.$ The longest element of
$E_7$ is the application
\begin{align*}
\R^8 &\to \R^8 \\
(x_1, \cdots, x_6, x_7,x_8) &\mapsto (-x_1, \cdots, -x_6, x_8,x_7),
\end{align*}
and its absolute length are equal to 7 and 63, respectively. We can
write the longest element as the composition of the reflections
$s_{r_1}, s_{r_2}, \ldots, s_{r_7}.$ Hence, the upper bound for the
Hofer-Zehnder capacity of a regular coadjoint orbit of $E_7$ is
$$
c_{\operatorname{HZ}}(\mathcal{O}_\lambda, \w_\lambda)\leq
2\lambda_2+2\lambda_4+2\lambda_6+\lambda_8-\lambda_7=2\lambda_2+2\lambda_4+2\lambda_6-2\lambda_7
$$
The highest root is $-e_7+e_8$ and
$$
c_{\operatorname{HZ}}(\mathcal{O}_\lambda, \w_\lambda)\geq
\max\Bigl\{2\lambda_6-2\lambda_7,\,\dfrac{1}{2}(\lambda_1+\cdots+\lambda_6-4\lambda_7)\Bigr\}
$$
Finally, for $E_6$ a system of simple roots is $\{\alpha_1, \cdots,
\alpha_6\}.$ The longest element of $E_6$ has absolute length and
length equal to 4 and 36, respectively, and it can be written as the
composition of the reflections $s_{t_1}, s_{t_2}, s_{t_3}$ and
$s_{t_4}$ where
\begin{align*}
&t_1=-e_2+e_3,
\,t_2=-e_1+e_4,\,t_3=\dfrac{1}{2}(e_1+e_2+e_3+e_4+e_5-e_6-e_7+e_8)\\
&t_4=\dfrac{1}{2}(-e_1-e_2-e_3-e_4+e_5-e_6-e_7+e_8)
\end{align*}
The upper bound for the Hofer-Zehnder capacity is given by
\begin{align*}
c_{\operatorname{HZ}}(\mathcal{O}_\lambda, \w_\lambda)\leq
&-\lambda_1-\lambda_2+\lambda_3+\lambda_4+\lambda_5-\lambda_6-\lambda_7+\lambda_8\\
=&-\lambda_1-\lambda_2+\lambda_3+\lambda_4+\lambda_5-3\lambda_6
\end{align*}
The highest root is $\dfrac{1}{2}(e_1+e_2+e_3+e_4+e_5-e_6-e_7+e_8)$
and the lower bound for the Hofer-Zehnder capacity is
\begin{align*}
c_{\operatorname{HZ}}(\mathcal{O}_\lambda, \w_\lambda)\geq
\lambda_5-\lambda_6-\lambda_7+\lambda_8=\lambda_5-3\lambda_6
\end{align*}

\subsection*{Type $F$}

A system of simple roots for $F_4$ is
$$
S=\{\alpha_1=e_2-e_3, \alpha_2=e_3-e_4, \alpha_3=e_4,
\alpha_4=\dfrac{1}{2}(e_1-e_2-e_3-e_4)\}.
$$
The Dynkin diagram of $F_4$ is shown in Figure \ref{F}.

\begin{figure}[h!]
\centering
\includegraphics[scale=0.7]{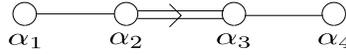}
\caption{Dynkin diagram of $F_4$}\label{F}
\end{figure}

The longest reflection $w_0$ in $F_4$ as a reflection of $\R^4$ is
\begin{align*}
\R^4 &\to \R^4 \\
(x_1, x_2, x_3, x_4) &\mapsto (-x_1, -x_2, -x_3, -x_4)
\end{align*}
We have that $l_T(w_0)=4, l(w_0)=24$ and
$$
w_0=s_{t_1}s_{t_2}s_{t_3}s_{t_4}
$$
where
\begin{align*}
t_1=e_1+e_2,\, t_2=e_1-e_2,\, t_3=e_3+e_4,\, t_4=e_3-e_4
\end{align*}
The upper bound for the Hofer-Zehnder capacity of a regular
coadjoint orbit of typer $F_4$ is
$$
c_{\operatorname{HZ}}(\mathcal{O}_\lambda, \w_\lambda)\leq
2\lambda_1+2\lambda_3
$$
The longest root of $F_4$ is $\rho=e_1+e_2.$ The lower bound for the
Hofer-Zehnder capacity is
$$
c_{\operatorname{HZ}}(\mathcal{O}_\lambda, \w_\lambda)\geq
2\lambda_1
$$

\subsection*{Type $G$}

Finally, a system of simple roots for $G_2$ is
$$
S=\{\alpha_1=e_1-2e_2+e_3, \alpha_2=e_2-e_3\}\subset \R^3
$$
and Dynkin diagram shown in Figure \ref{G}

\begin{figure}[h!]
\centering
\includegraphics[scale=0.7]{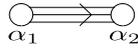}
\caption{Dynkin diagram of $G_2$}\label{G}
\end{figure}

We write
$$
w_0=s_{t_1}s_{t_2}
$$
where
$$
t_1=e_2-e_3,\,\, t_2=2e_1-e_2-e_3
$$
Hence,
$$
c_{\operatorname{HZ}}(\mathcal{O}_\lambda, \w_\lambda)\leq
\dfrac{2}{3}(\lambda_1-\lambda_2-2\lambda_3)=\dfrac{2}{3}(3\lambda_1+\lambda_2)
$$
The highest root is $\rho=2e_1-e_2-e_3,$ and
$$
c_{\operatorname{HZ}}(\mathcal{O}_\lambda, \w_\lambda)\geq
\dfrac{2}{3}(2\lambda_1+\lambda_2)
$$

All the bounds for the Hofer-Zehnder capacity of coadjoint orbits
are summarized in the following table

\begin{center}
\begin{tabular}{c|c|c}
$G$ & Lower bound & Upper bound \\ \hline
& & \\
$U(n)$ & $\frac{1}{2}\sum_{i=1}^n |\lambda_i-\lambda_{n-i+1}|$ &
$\frac{1}{2}\sum_{i=1}^n
|\lambda_i-\lambda_{n-i+1}|$ \\
$Sp(2n)$ & $\lambda_1+\cdots+\lambda_n$ &
$\lambda_1+\cdots+\lambda_n$ \\
$SO(n)$ & & \\
$n=4m$ & $\lambda_1+\cdots+|\lambda_n|,\,\,\, 2\lambda_1$ &
$2\lambda_1+2\lambda_3+\cdots+2\lambda_{n-1}$ \\
$4m+1$ & $\lambda_1+\cdots+\lambda_n,\,\,\, 2\lambda_1$ &
$2\lambda_1+2\lambda_3\cdots+2\lambda_{n-1}$ \\
$4m+2$ & $\lambda_1+\cdots+\lambda_{n-1},\,\,\, 2\lambda_1$ &
$2\lambda_1+2\lambda_3+\cdots+2\lambda_{n-2}$ \\
$4m+3$ & $\lambda_1+\cdots+\lambda_{n},\,\,\, 2\lambda_1$ &
$2\lambda_1+2\lambda_3\cdots+2\lambda_{n}$ \\
$E_6$ & $\lambda_5-\lambda_6-\lambda_7+\lambda_8$ &
$\lambda_3+\lambda_4+\lambda_5-\lambda_1-\lambda_2-3\lambda_6$ \\
$E_7$ &
$\frac{1}{2}(\lambda_1+\cdots-4\lambda_7),\,\,\,2\lambda_6-2\lambda_7
$&
$2\lambda_2+2\lambda_4+2\lambda_6-2\lambda_7$ \\
$E_8$  & $\frac{1}{3}(\lambda_1+\cdots+5\lambda_8),\,\,\,2\lambda_8
$&
$2\lambda_2+2\lambda_4+2\lambda_6+2\lambda_8$ \\
$F_4$ & $2\lambda_1$ & $2\lambda_1+2\lambda_3$ \\
$G_2$ & $\frac{2}{3}(2\lambda_1+\lambda_2)$&
$\frac{2}{3}(3\lambda_1+\lambda_2)$
\end{tabular}
\end{center}

\begin{rk}

Note that regardless of our bounds being sharp or not, we always get
the following inequality
$$
\dfrac{2}{3}\sum_{k=1}^r \langle \lambda, \check{\alpha}_k \rangle
\leq c_{\operatorname{HZ}}(\mathcal{O}_\lambda, \w_\lambda)\leq
\sum_{k=1}^r \langle \lambda, \check{\alpha}_k \rangle
$$
\end{rk}

\section{Acknowledgments}

I would like to thank Yael Karshon and Leonid Polterovich for useful
discussions. This research is supported by the Israel Science
Foundation grants $178/13$ and $1380/13.$

\renewcommand{\refname}{Bibliography}
\bibliographystyle{plain}
\bibliography{biblo}
\nocite{*}

\end{document}